\documentclass[12pt]{amsproc}
\usepackage{amssymb, url, amsmath, amsthm}

\renewcommand{\leq}{\leqslant}

\renewcommand{\geq}{\geqslant}
\renewcommand{\ge}{\geqslant}

\newtheorem{theorem}{Theorem}
\newtheorem{lemma}[theorem]{Lemma}
\newtheorem{corollary}[theorem]{Corollary}

\theoremstyle{definition}
\newtheorem{definition}[theorem]{Definition}

\newtheorem{problem}{Problem}
\newtheorem{example}[theorem]{Example}

\title[Orbit theory and locally finite permutations]{Orbit theory, locally finite permutations and Morse arithmetic}
 \begin{document}
\author{A.~M.~Vershik}
\address{St.~Petersburg Department of Steklov Institute of Mathematics and
Max Plank Inst. Bonn} \email{vershik@pdmi.ras.ru}
\thanks{Partially supported by NSh
2460.2008.1112, and RFBR 08-01-00379-a}
\subjclass[2010]{Primary }
\maketitle 

\rightline{\it \textbf{To the memory of my wife Rita}}

\bigskip
\begin{abstract}
The goal of this paper is to analyze two measure preserving
transformation of combinatorial and number-theoretical origin
from the point of view of ergodic orbit theory. We study the
Morse transformation (in its adic realization in the group
$\textbf{Z}_2$ of integer dyadic numbers, as described by the
author [{\sl J. Sov. Math.} {\textbf{28}}, 667--674 (1985);
{\sl St. Petersburg Math. J.} {\textbf{6}} (1995), no. 3,
529--540]) and prove that it has the same orbit partition as
the dyadic odometer. Then we give a precise description of time
substitution of the odometer, which produces the Morse
transformation. It is convenient to describe this time
substitution in the form of random re-orderings of the
group~$\mathbb Z$, or in terms of random infinite permutations
of the group~$\mathbb Z$. We introduce the notion of {\it
locally finite permutations (LFP) or locally finite bijection (LFB),
and uniformly locally finite time
substitution (ULFTS)} for the group~$\mathbb Z$ (and for
all amenable groups). Two automorphisms which have the same
orbit partitions are called {\it allied} if the time
substitution of one to another is ULFTS.
Our main result is that the Morse transformation
and the odometer are allied. The theory of random infinite
permutations on the group~$\mathbb Z$ (and on more general
groups) is as strong as the ergodic theory of actions of the
group. The main task in this area is the investigation of
infinite permutations, and measures on the space of infinite
permutations, as well as the study of linear orderings
on~$\mathbb Z$. The class of locally finite permutations is a
useful class for such an analysis.
\end{abstract}

\maketitle

\section{Introduction}

We consider in detail an example of time changing of a measure
preserving transformation, namely the Morse transformation as a
time-substitution of the odometer, which is the simplest
ergodic transformation. The odometer is the operation of adding
unity,~$Tx=x+1$, in the additive group~$\textbf{Z}_2$ of
integer dyadic numbers. The main result is the calculation of a
random re-ordering on the orbits of the odometer which presents
the Morse transformation. The old observation of the
author~\cite{V1,VL} is that these two transformations have the
same orbit partition in their natural realization in the
space~$\textbf{Z}_2$; see also the recent paper~\cite{VS}.
Until now no explicit form of time substitution for a
non-trival pair of automorphisms with the same orbit partitions
was known. The odometer and the Morse transformation are
perhaps the first useful example. Notice that the Morse
transformation is a two-point extension of the odometer, and
has non-discrete part in the spectrum, but a more delicate
relation between them will become clear from the description of
the corresponding time substitution. Namely, we prove that they
are {\it allied} in the sense defined below.

In Section~2 we recall the main facts of orbit theory (Dye's
theorem in particular), and formulate the general problem about
possible reduction of ergodic theory to the theory of random
permutations of the naturals numbers. Sections~3 and~4 are
devoted to the Morse transformation and its realization as an
adic transformation. This presentation allows us to prove that
the orbit partitions of the odometer and of the Morse
transformation are the same. The main result, which we explain
in Section~5, consists of a description of the algorithm for
the time substitution or re-ordering of the orbits of the
odometer to the orbits of the Morse transformation. Finite
substitutions of the set~$\{2^n,2^n+1, \dots,2^{n+1}-1\}$ -- so
called Morse substitutions -- play the key role in this
algorithm. We emphasize the opening up of the possibility to
study new models of measure preserving transformations, and new
related tools. By this we mean that studying the group of
permutations of the group~$\mathbb Z$ and the space of linear
orders on~$\mathbb Z$, as well as invariant measures on those
spaces. For example, we define a new relationship between two
measure preserving transformations or between two actions of an
amenable group with the same orbit partitions as follows. We
say that the actions {\it are allied} if there is a time
substitution of one action to another which is locally finite
with the fixed set of finite permutations,
which is the same for almost every orbit. In this case we called
time substitution as {\it uniformly locally finite time substitution.}
(ULFTS).
 We prove that {\it the odometer and the Morse
transformation are allied}. More precisely, the
corresponding measure on the group of permutation of~$\mathbb
Z$ is concentrated on the permutations which almost preserve
increasing sequences of intervals and finite permutations of intervals
depends on the length of interval only. A more convenient language
here is that of {\it re-ordering} of the group~$\mathbb Z$. In
general, we can say that {\it ergodic theory (for actions of
the group~$\mathbb Z$) can be considered as the theory of
random permutation of the naturals, or as the theory of random
linear orderings on the naturals.} The group acts on this
space, and this action is universal in the sense that any
action (up to isomorphism) can be realized in this way. The
re-ordering of the orbit show us how different are the two
transformations, or, how random is the re-ordering of their
orbits. This point presumably is useful for automorphisms with
positive entropy (in order to measure the difference between
Bernoulli and non-Bernoulli automorphisms). The very special
and ingenious structure ({\it of locally finite ordering}) in
the case of the Morse transformation is also interesting as a
new source of measures in the group of infinite permutations.

The similar definitions of locally finite bijection and
allied actions can be done for an arbitrary amenable group.

\section{Ergodic theory as analysis of infinite permutations}

\subsection{Dye's theorem and orbit theory}
The simplest ergodic transformation is undoubtedly the~2-adic
odometer (or ``adding machine''):
\[
T:x\mapsto x+1, \quad x \in \textbf{Z}_2,
\]
where~$\textbf{Z}_2$ is the additive group of~2-adic integers
equipped with Haar measure~$m$. Indeed,~$T$ has discrete
spectrum, comprising the group of all roots of unity of
order~$2^n, n=0,1 \dots$. Although the structure of this
automorphism is very simple, its orbit partition is universal
in the class of ergodic measure-preserving transformations, as
mentioned above.

\begin{theorem}[H. Dye~\cite{D}]\label{dyetheorem}
For each ergodic measure preserving transformation~$M$ of the space~$(X,m)$,
there exists a transformation~$S$ of the same space which is metrically isomorphic to
the~$2$-odometer~$T$ and
which has the same orbit partition as~$M$, so that~$M$ is a
\emph{time change} of~$S$:
\[
Mx=S^{t(x)}x,
\]
where~$t(\cdot)$ is a~${\mathbb Z}$-valued measurable function on~$X$, and {\it vice versa}, there exists a
measurable function~$n(\cdot)$
such that~$Sx=M^{n(x)}x$.
\end{theorem}

Isomorphism between~$T$ and~$S$ means the existence of an
invertible measure-preserving map~$V$ such
that~$VTV^{-1}=S$,~$V:X \rightarrow X,$ and~$V\mu=\mu$. So, by
Dye's theorem, each ergodic automorphism is isomorphic to an
automorphism of~$\textbf{Z}_2$, preserving Haar measure, and
with the same orbit partition as the odometer.

The functions~$t(\cdot), n(\cdot)$ are called
\emph{time-substitutions} or \emph{jump-functions}.

So, the odometer has a universal orbit partition: the orbit
partition of any ergodic transformation~$M$ is metrically
isomorphic to the orbit partition of the odometer~$T$. This
partition is a standard hyper-finite countable homogeneous
ergodic partition, that is the union of the decreasing ergodic
sequence of the finite~$2^n, n=1, \dots$ homogeneous
partitions\footnote{We will not dwell on the interesting
history of this theorem, which goes back to J. von Neumann and
Murray~\cite{vN} (where it is phrased in terms of the
uniqueness of type-$II_1$ hyper-finite factors), and to several
papers of H.Dye~\cite{D}. In~\cite{V} the author proved a
lacunary theorem for homogeneous sequences of measurable
partitions, and R. Belinskaya~\cite{B} used these as the main
ingredient for a new proof of Dye's theorem (not
known to us at that time), see also the later paper~\cite{HK}. The
final result in this direction was proved by Ornstein and
Weiss~\cite{OW} and Connes, Feldman, and Weiss~\cite{CFW}: all
ergodic measure preserving actions of countable amenable groups
have hyper-finite (or ``tame'') orbit partitions, isomorphic to
the orbit partition defined by the odometer. Together with the
previous results, this gives a characteristic property of
measure-preserving actions of amenable groups.}.

By definition, the orbit partition of the odometer~$T$ is the
partition into the cosets of the subgroup~$\mathbb Z$ in the
group~$\textbf{Z}_2$. We will prove that it
coincides~$\pmod{0}$ with the \emph{tail partition} -- that is,
the partition into the cosets of the group~$\sum_1^{\infty}
({\mathbb Z}/2)$ in the group~$\prod_1^{\infty} ({\mathbb
Z}/2)$. There is a difference between these two partition on
the countable set~$\mathbb Z$ due to the fact that positive and
negative integers belong to different tail classes, namely the
classes of~$\{0\}^\infty$ and~$\{1\}^{\infty}$ in the
group~$\textbf{Z}_2$. Both the groups~$\textbf{Z}_2$
and~$\prod_1^{\infty} ({\mathbb Z}/2)$ may be equipped with
Haar measures, and become metrically isomorphic as measure
spaces with those partitions. We will prove that the orbit
partition of the Morse automorphism in its adic realization
also coincides~$\pmod{0}$ with the tail partition, and
consequently there is a reversible time substitution which
brings the Morse automorphism to the odometer. The goal of the
paper is to study the arithmetic properties of that time
substitution, and the so-called Morse arithmetic.

\subsection{General theory of time
substitutions}\label{sectionongeneraltimesubstitutions}

Suppose a countable group~$G$ acts as measure-preserving
transformations on the measure space~$(X, \mu)$. The orbit of
the point~$x\in X$ is the set~$\{gx; g\in G\}$. The partition
of the space~$(X, \mu)$ into the orbits of the group~$G$ is called
the orbit (or trajectory) partition, and denoted by~$\xi(G;X)$
or~$\xi(G)$ if the action and space are fixed. In the case of
the group~$G=\mathbb Z$, we can write~$\xi(T)$, where~$T$ is
the generator of the action of the group~$\mathbb Z$.

Assume that two measure preserving ergodic automorphisms~$T$
and~$M$ of the Lebesgue space~$(X,m)$ have the same orbit
partitions,
\[
\xi(T)=\xi(M)=\xi.
\]
Consider the time-change function~$t(\cdot)$ from above,
defined by the formula:
\[
Mx=T^{t(x)}x,
\]
which exists because of the coincidence of orbit partitions.
The function~$t:X\to\mathbb Z$ is called a \emph{time
substitution} function from the automorphism~$T$ to the
automorphism~$M$ or the ``function of jumps from~$T$ to~$M$''.

Consider also the \emph{complete time change}
function~$(x,k)\mapsto t(x,k)$, defined by
\[
M^k x=T^{t(k,x)}x.
\]
It is easy to express the function~$t(k,x)$ using values of the
function~$t(x)$ on the same orbit as follows. We record the
following formulas for~$t(x,k)$:
\begin{eqnarray*}
t(0,x)&\equiv&0,\\
t(1,x)&=&t(x),\\
t(k,x)&=&\sum_{i=0}^{k-1} t(M^{i} x)\quad\mbox{for }k>0,\mbox{ and}\\
t(k,x)&=&-\sum_{i=0}^{|k|-1}t(M^{i} x)\quad\mbox{for }k<0.
\end{eqnarray*}
For example,
\[
M^{-1}x=T^{-t(M^{-1}x)} x
\]
and
\[
t(-1,x)=-t(M^{-1}x).
\]
It is clear from the definition that~$M$ and~$T$ have the same
orbit partition if the set of values of function~$t(x,\cdot)$
when~$k$ run over~$\mathbb Z$ coincides with~$\mathbb Z$ for
almost every~$x$, or in other words~$\{t(M^k x)\mid k\in
\mathbb Z\}=\mathbb Z$ for a.a.~$x$.

It is natural to look at the function~$t(\cdot,\cdot)$ as a map
from~$X$ to the group of all permutations of the integers --
${\frak S}^{\mathbb Z}$. We denote this map by~$\Theta_M$, so
\[
\Theta_M: x \mapsto \{k \mapsto t(x,k)\}\in {\frak S}^{\mathbb
Z}.
\]
Evidently for almost all~$x$ the bijection~$\Theta(x)$
of~$\mathbb Z$ is an element of~${\frak S}^{\mathbb Z}$.

We will call the permutation~$\{k\rightarrow
t(k,x)\}=\Theta(x)\in {\frak S}^{\mathbb Z}$ the \emph{time
substitution from~$T$ to~$M$ at the point~$x$}, or for brevity
(if it is clear from the context what the maps~$T$ and~$M$ are)
-- the \emph{substitution at the point~$x$}. It is important
not to confuse this permutation with the time
substitution~$t(x)$, which was defined above. The formulas
above give the links between these two objects.

Define a subgroup~${\frak S}_0$ of the group~${\frak
S}^{\mathbb Z}$ by
\[
{\frak S}_0^{\mathbb Z}\subset {\frak S}^{\mathbb Z}: {\frak
S}_0^{\mathbb Z}=\{g\in {\frak S}^{\mathbb Z}\mid g(0)=0\}.
\]
Since~$t(0,x)\equiv0$, the image of~$X$ under the
map~$\Theta_M$ lies in the subgroup~${\frak S}_0^{\mathbb Z}$.
Thus each point~$x \in X$ is sent to an infinite
permutation~$\Theta(x):k\mapsto t(x, k)\in \mathbb Z$, an
element of the group~${\frak S}^{\mathbb Z}$ of all infinite
permutations of~$\mathbb Z$.

Define an action of the group~$\mathbb Z$ on~${\frak
S}_0^{\mathbb Z}$ by letting the generator~$D$ of~$\mathbb Z$
act according to the formula
\[
D(g)(k)=g(k+1)-g(1),
\]
($D$ is the ``modified shift''). The map~$D$ is a bijection of
the subgroup~${\frak S}_0^{\mathbb Z}$ onto itself (but is not
a group isomorphism). The identity permutation~$Id$ and the
reflection~$-Id$ are both fixed points of the map~$D$. We have
defined an action of~$\mathbb Z$, that is a dynamical system,
on a subset of
\[
{\frak S}_0^{\mathbb Z}\subset {\frak S}^{\mathbb Z}\mid{\frak
S}_0^{\mathbb Z}=\{g\in {\frak S}^{\mathbb Z}\mid g(0)=0\}.
\]

We have defined a map from our system~$(X,m,M)$ to the
group~${\frak S}_0^{\mathbb Z}$ of all permutations of the
integers with zero as a fixed point,
\[
\Theta_M:X \rightarrow {\frak S}_0^{\mathbb Z},
\]
defined by
\[
X\ni x \mapsto \Theta_M(x) \equiv \{k \mapsto t(k,x)\mid k\in \mathbb Z \}.
\]
The image of the measure~$m$ on~$X$ under~$\Theta_M$ is a
measure~$\Theta_M m=\mu_M$ on the group~${\frak S}_0$. We
define a new dynamical system using this as follows.

\begin{theorem}
The map~$\Theta_M$ is a homomorphism of the ergodic
transformation~$M$ of the space~$(X,m)$ to the ergodic
transformation~$D$ on the space~$({\frak S}^{\mathbb
Z}_0,\mu_M)$.

If for almost all pairs~$(x,y)$ with~$x,y\in X$ we
have~$t(\cdot,x)\ne t(\cdot,y)$, then~$\Theta_M$ is an
isomorphism onto the image of the triple~$(X,m,M)$ and
\[
({\frak S}^{\mathbb Z}_0,\mu_M, D).
\]
\end{theorem}

Thus we obtain (in the non-degenerate cases) a new model for
the study of the automorphism~$M$ (relative to the
automorphism~$T$). The measure~$\mu_M$ is interesting itself as
a natural example of a measure on the group of all permutations
of the integers (more generally, for a~$G$-action, a measure on
the group of permutations of~$G$).

\subsection{The space of linear orderings of~$\mathbb Z$}

In what follows it is more convenient to think of the
time-substitution in terms of linear orders on~$\mathbb Z$;
more precisely \emph{re-orderings} of the integers~$\mathbb Z$.
The image of the usual order on the group~$\mathbb Z$ under the
time substitution~$\Theta_M(x)$ can be represented as follows:
\[
\cdots\negthinspace\rightarrow\negthinspace
t(-2,x)\negthinspace\rightarrow\negthinspace t(-1,x)
\negthinspace\rightarrow\negthinspace
t(0,x)\negthinspace=\negthinspace0
\negthinspace\rightarrow\negthinspace
t(1,x)\negthinspace=\negthinspace t(x)
\negthinspace\rightarrow\negthinspace
t(2,x)
\rightarrow
\negthinspace\cdots.
\]
This is a \emph{linear order of type~$\mathbb Z$ on the
group~$\mathbb Z$}, depending on~$x$. We will use the following
notation from combinatorics:~$a\lessdot b$ means that~$b$ is
the immediate successor of~$a$, or~$b$ immediately follows~$a$
with no intermediate elements. Using this, we can write
\[
\cdots t(-2,x) \lessdot t(-1,x)\lessdot t(0,x)=0 \lessdot t(x)=t(1,x) \lessdot t(2,x) \cdots.
\]
Consider the space~$\mathcal T$ of all linear orders
on~$\mathbb Z$ of type~$\mathbb Z$, and equip that space with
the natural weak topology and corresponding Borel structure.
The group of shifts~$\{S^n\mid n \in \mathbb Z\}$ acts
on~$\mathcal T$, and we can consider the shift-invariant Borel
measures on the space~$\mathcal T$. The \emph{set of
triples}~$({\mathcal T}, S,\mu)$ is once again a universal
model in ergodic theory, as was the previous model of the group
of all permutations. The formulas of the following lemma and
its corollary show how to express this action.

\begin{lemma}\label{lemmaA}
For~$k\geq 0,$
\[
t(k,Mx)=t(k+1,x)-t(x)\quad(=\sum_{i=1}^k t(M^i
x)),
\]
for~$k<0$, and
\[
t(k, Mx)=t(k-1,x)-t(x)\quad(=-\sum_{i=1}^{|k|} t(M^{-i}
x)).
\]
\end{lemma}

\begin{corollary}
$t(M^k x)=t(k+1,x)-t(k,x).$
\end{corollary}

Lemma~\ref{lemmaA} shows that the order
\[
t(k,x)\lessdot t(k+1,x), k\in \mathbb Z,
\]
is sent to the order
\[
t(k,Mx)\lessdot t(k+1,Mx)(=t(k+1,x)-t(x)), k\in \mathbb Z,
\]
or simply
\[
t(k,x)-t(x) \lessdot t(k+1,x)-t(x)),k\in \mathbb Z.
\]

This means that the new linear order induced by the
automorphism~$M$ is simply the shift (translation) of the
previous order. So the induced action of the group~$\mathbb Z$
on the space~$\mathcal T$ is defined independently of the
automorphism~$M$ (recall that the function~$t(\cdot)$ takes on
all integer values). Thus the \emph{action of $\mathbb Z$ on
the space of linear orders~$\mathcal T$ is simply the action by
the shifts}. The usual order~$k\lessdot k+1,  k\in \mathbb Z$
and the opposite order~$k\lessdot k-1,  k\in \mathbb Z$ are
fixed points of this action.

We are interested in the \emph{shift-invariant measures
on~$\mathcal T$}. For a given automorphism~$T$ we can identify
each automorphism~$M$ which has the same orbit partition as~$T$
with the shift invariant measure on the space~$\mathcal T$.
This measure is concentrated on the set of re-orderings of the
orbits of~$T$.

Thus all ergodic triples~$(X,M,m)$ up to isomorphism can be
realized as the triple~$(\mathcal T, S, \mu)$, where~$S$ is the
shift on the space~$\mathcal T$ and~$\mu$ is a shift invariant
measure on the~$\mathcal T$. Remember that this isomorphism
depends on the automorphism~$T$, so this model as before can
only give relative invariants of~$M$ with respect to~$T$.

\begin{problem}
How does the class of invariant measures depend on the
automorphism~$T$?
\end{problem}

\subsection{The notion of locally finite permutation and locally finite
ordering}\label{sectiononlocallyfiniteorderings}

Now we define a special class of permutations (or bijections)
of the group~$\mathbb Z$ (and, more generally, of countable
amenable groups) and in parallel the corresponding special
class of linear orders on the group~$\mathbb Z$.

\begin{definition}
Locally finite bijections are defined as follows.
\begin{enumerate}
\item A bijection~$L:\mathbb Z\to\mathbb Z$ is called
    \emph{locally finite} (LFB) if there exists an increasing
    sequence of intervals~$I_1\subset I_2 \subset \cdots$,
    with
    \[
    \bigcup_{n\ge1}I_n=\mathbb Z,
\]
    with the property that for each~$\epsilon>0$ there
    exists~$N=N(\epsilon)$ such that the
    intervals~$I_n,n>N$ are~$L$-invariant up
    to~$\epsilon>0$\footnote{This means that~$|I_n\Delta
    L(I_n)|<\epsilon|I_n|$.}.
\item More generally, let~$G$ be an arbitrary countable
    amenable group, and~$L:G\to G$ be a bijection of $G$
    onto itself. Then~$L$ is called \emph{locally finite}
    if there is an increasing exhaustive sequence of finite
    sets~$\{I_n\}, \bigcup_n I_n=G$, where each~$I_n$ is
    a~$\frac{1}{n}$-F{\o}lner set, (for some fixed choice
    of generators of~$G$), with the property that for
    each~$\epsilon>0$ there exists some~$N=N(\epsilon)$
    such that for all~$n>N$ the sets~$I_n$
    are~$L$-invariant up to~$\epsilon$.
\end{enumerate}
\end{definition}

The property of being a locally finite bijection does not
depend on the choice of generators for the group, and the class
of locally finite bijections generates a subgroup of the group
of all bijections on the group. To my knowledge nothing is
known about this subgroup, and it would be interesting to
investigate the algebraic structure of this subgroup.

It is easy to give examples of orderings which are not locally
finite, but the locally finite class is most natural in
analysis. We use this notion in the following paragraphs. The
parallel definition for linear orderings on the group~$\mathbb
Z$ is the following.

\begin{definition}
A linear ordering~$\prec$ on the group~$\mathbb Z$  (of
type~$\mathbb Z$) is called \emph{locally finite}(LFO) if there
exists an increasing sequence of finite intervals~$I_n=(a_n,
b_n), \bigcup I_n=\mathbb Z$, with a linear order~$\prec_n$ on
each~$I_n$ such that the following holds.
\begin{enumerate}
\item The linear order~$(\mathbb Z,\prec)$ is a limit
    (stabilization in the natural sense) of~$(I_n,\prec_n)$
    when~$n$ tends to infinity.
\item For each~$\epsilon>0$ there exists~$N=N(\epsilon)$
    with the property that for~$n>N$ the restriction
    of~$\prec_n$ to the interval~$I_m,m<n$ coincides
    with~$\prec_m$ for all elements of~$I_m$
    except~$\epsilon\cdot|I_m|$\footnote{In the sequel we
    will have a stronger condition, that the restriction
    of~$\prec_n$ to~$I_m$ differs from~$\prec_m$ at most on
    two points of~$I_m$.}.
\end{enumerate}
\end{definition}

For the group~$\mathbb Z$ the notion of locally finite ordering
is consistent with the definition of locally finite bijection
above. Namely, it is easy to check from the definition that
if~$L$ is a locally finite bijection then the~$L$-image of~$>$
(that is, the order defined by~$n\prec m \iff L^{-1}n>L^{-1}m$)
is locally finite, and {\it vice versa}.

It is not difficult to prove (see \cite{V2}) the following fact:
\begin{theorem}
Let $S$ is ergodic automorphism which has the same orbit partition as
odometer T. Then there exists time substitution of almost all orbits of $T$
to the orbits of $S$ which is locally finite for almost all orbits.
The growth of the lengths of the corresponding intervals $I_n$ depends on so called
scale of automorphism $S$.
\end{theorem}
Remark that by Dye's theorem and previous assertion each ergodic measure preserving transformation is
isomorphic to another measure preserving transformation which has the same orbit as odometer
and for which time substitution is locally finite for almost all orbits.

For Morse automorphism we will prove much more strong assertion
about time substitution.
\begin{definition}
Suppose that for two measure preserving ergodic transformations $T,M$ with the same orbits, and
the locally finite time substitution from $T$ to $M$ is defined with permutations of each intervals $I_n$
$n=1,2 \dots$ (see definition) which are depended on $n$ but are the same for almost all orbits.
We call such time substitution  as uniformly locally finite -ULFTS.
(This is symmetric relation with respect to $T,M$. In this case we say that $T,M$ are allied
transformations.
\end{definition}
We will call that Morse automorphism and odometer are allied, the time substitution
 has the exponential growth of length of interval and the finite permutation
for each length of the interval  $I_n$ is so called the Morse permutation
which is defined for each interval of length $2^n,n=1,2\dots$ and is the same for almost all orbits.

\section{The odometer and the Morse automorphism}

\subsection{Orbit partition of the odometer}

Assume that the space~$X$ is~$\textbf{Z}_2$, the integer dyadic
numbers with Haar measure~$m$ and~$T$ is the transformation of
adding unity in the group~$\textbf{Z}_2$,
\[
T:x\rightarrow x+1.
\]
Notice that as a topological space (with pro-finite topology),
and as a measure space (with Haar measure), the
group~$\textbf{Z}_2$ is the same as the space of
group~$\prod_1^{\infty} {\mathbb Z}/2$ (with
the~$(\frac12,\frac12)$ Bernoulli measure). Thus we can
consider the natural action of the
group~$\sum_1^{\infty}{\mathbb Z}/2$ on~$\textbf{Z}_2$.

Let us consider the element~$x \in \textbf{Z}_2$ as a
sequence~$(x_1, x_2,\dots )$ with~$x=0,1$, of the dyadic
decomposition. Two sequences~$\{x_k\}, \{y_k\}$ are called
\emph{cofinal} if~$x_k=y_k$ for all~$k>k_0(x,y)$. Finite
sequences (that is, sequences which are cofinal
with~$(0)^{\infty}$ correspond to natural rational integers,
sequences which are cofinal to~$(1)^{\infty}$ correspond to
negative rational integers. Each equivalence class of cofinal
sequences is an orbit of the action of the
group~$\sum_1^{\infty}{\mathbb Z}/2$.

\begin{lemma}\begin{enumerate}
\item The orbit partition of odometer~$T$ is the partition into
    the cosets of the subgroup of integer rational numbers~${\mathbb Z}\subset\textbf{Z}_2$,
    and in terms of dyadic decompostion:
\item The orbit partition~$\xi(T)$ of the odometer~$T$ on the
    set~$\textbf{Z}_2\setminus \mathbb Z$ coincides with the partition into
    cofinal classes (or with the orbit partition of the group~$\sum_1^{\infty}{\mathbb Z}/2$.
    So these two partitions are the same~$\pmod{0}$:
    \[
    \xi(T)=\xi(\sum_1^{\infty} {\mathbb Z}/2) \quad \mbox{mod 0}.
    \]
    The integers~$\mathbb Z$ generate one~$T$-orbit, but decompose into
    two classes of cofinality -- cofinal with~$(0)^{\infty}$ and~$(1)^{\infty}$.
\end{enumerate}
\end{lemma}

\begin{proof}
If~$n\in{\mathbb Z}_+$ then evidently~$T^nx=x+n$ and~$x$ are
cofinal; if~$x\ne 0$, then~$x-1$ and~$x$ are cofinal.
Consequently, if~$x\in\mathbb Z$ and~$n\in\mathbb Z$,
then~$x+z$ and~$x$ are cofinal. If~$x,y$ are cofinal then there
exists~${z\in\mathbb Z}_+$ such that either~$x+z=y$ or~$y+z=x$.
The orbit~$\mathbb Z$ consists of two classes of cofinal
sequences -- one cofinal with~$0^{\infty}$, and one
with~$(1)^{\infty}$.
\end{proof}

We can identify~$T$-orbits of the generic point~$x$ of the
group~$\textbf{Z}_2$ with the group~$\sum_1^{\infty}{\mathbb
Z}/2$, so it provides a linear order of type~$\mathbb Z$ on the
group~$\sum_1^{\infty}{\mathbb Z}/2$.

\begin{example}
An interesting example of this nontrivial linear order is to
use the point~$x=-\frac{1}{3}\in\textbf{Z}_2$, which has dyadic
decompostion
\[
-\textstyle\frac{1}{3}=(10)^{\infty}
\]
(see below). Its~$T$-orbit is the set~$\{\frac{1}{3}- k\mid
k\in \mathbb Z\}$. On the other hand, the orbit is the class of
all sequences cofinal with~$(10)^{\infty}$, so it is possible
to order by type~$\mathbb Z$ this orbit, or {\it vice versa} to
identify the group~$\sum {\mathbb Z}/2$ which parameterizes
that class in the natural sense, with the group~$\mathbb Z$.
The reader can do this easily.
\end{example}

\begin{example}
Notice that the group~$S_{\infty}$ also acts on the
group
\[
\textbf{Z}_2\simeq\prod_1^{\infty}{\mathbb Z}/2
\]
as group permutations of the coordinates. However its orbit
partition is finer than the partition into classes of
cofinality, but does not coincide with it~$\pmod{0}$. The
so-called Pascal automorphism~\cite{V1} also acts on the group,
and has the same orbit partition as the action of~$S_{\infty}$-
infinite symmetric group.
\end{example}

If a measure-preserving transformation~$M$ of~$\textbf{Z}_2$
has the same orbit partition as the odometer~$T$, then our
formulas for~$M$ simplify. The function~$X \ni x \mapsto t(x)
\in \mathbb Z$ is
\[
Mx=T^{t(x)}x= x+t(x),
\]
(here addition is in the sense of the group~$\textbf{Z}_2$). We
can simplify other formulas. For example,
\[
M^k x=T^{t(k,x)}x=x+t(k,x) \pmod{\textbf{Z}_2},
\]
and so on.

At the same time, we believe that consideration of the Morse
transformation below is a typical consideration in orbit
theory. Namely, we will consider the well-known Morse system
from the point of view of the dyadic odometer (that is, the
relative invariants), and for this we will use the adic
realization of the Morse automorphism, and the corresponding
Morse permutations of the integers.

\subsection{Traditional definition of the Morse
automorphism}\label{sectiontraditionaldefinitionofmorse}

We recall the definition of the Morse automorphism. Consider
the alphabet~$\{0,1\}$ and define the \emph{Morse substitution}
by
\begin{eqnarray*}
\zeta(0)&=&01,\\
\zeta(1)&=&10.
\end{eqnarray*}
This is extended to all words in the alphabet~$\{0,1\}$ by
concatenation. The Thue--Morse sequence is a fixed point of the
substitution,
\[
u=u_0u_1u_2\cdots=\lim_{n\to\infty}
\zeta^n(0)=0110100110010110\ldots.
\]
The sequence~$u$ obeys the rule
\[
u[0,2^{n+1}-1] = u[0,2^n-1]\,  {{\bar u}[0,2^n-1]}\ \ \mbox{for}\ n\ge 0,
\]
where~$u[i,j]=u_i\cdots u_j$ and~${\bar u}[\medspace\dots]$
denotes changing~$0\leftrightarrow 1$
inside~$[\medspace\dots]$.

The sequence~$u$ is non-3-periodic (no words appear~$3$ times
repeatedly), and similarly is non-$4$-periodic and
non-$5$-periodic. The sequence is well known in combinatorics,
logic, and symbolic dynamics.

We next define the Morse automorphism. Consider the set of all
two-sided infinite sequences in the alphabet~$\{0,1\}$, all
sub-words of which are the sub-words of the infinite
Morse--Thue sequence~$u$\footnote{In other words, this
is the shift whose language coincides with the language defined
by~$u$.}; this is the same set as the weak closure of the
two-sided shifts of the Morse--Thue sequence~$u$. This is a
shift-invariant compact subset (in the weak topology)~$\mathcal
M$ of the space~$Y=\{0,1\}^{\mathbb Z}$. The (left) shift
on~$\mathcal M$ is by definition the \emph{Morse automorphism}
of~$M$ as a topological space with action of $\Bbb Z$.

We recall the main properties of the Morse automorphism (see
the earlier papers including~\cite{Ke}, and more recent
treatments~\cite{VS}).

\begin{theorem}
The automorphism~$M$ is a minimal, uniquely ergodic, uniformly
recurrent, transformation of~$\mathcal M$. The spectrum of the
corresponding unitary operator~$U_M$ on~$L^2({\mathcal M},\mu)$
(where~$\mu$ is the unique invariant measure) is mixed, and
contains a discrete part which is dyadic (and the same as for
the~$2$-odometer) and a continuous part which is singular with
respect to Lebesgue measure.
\end{theorem}

\section{The Morse transformation and its orbit partition}

\subsection{Adic realization of the Morse
transformation}\label{subsectA}

Consider the group~$\textbf{Z}_2$, the compact abelian group of
the~$2$-adic integers under addition. It is clear that as a
measure space this is isomorphic to the countable product of
copies of the group~${\Bbb Z}/2$:~$\textbf{Z}_2\cong\{0,1\}^{\mathbb
N}$. The odometer transformation~$T$ is an \emph{adic
transformation} by definition, with respect to the natural
partial (lexicographic) order on the group of~$2$-adic
integers:
\[
T: x\mapsto x+1.
\]
We want to compare the Morse automorphism and the~$2$-odometer.
For this we use the so called \emph{adic realization} of the
Morse automorphism. The general definition of an adic
transformation requires an~$\mathbb N$-graded graph (in our
case, a graph with the two vertices~$0$ and~$1$ on all levels,
and simple edges between any two vertices of adjacent levels),
and with a linear order on the set of edges which end to a
given vertex. For the odometer, the order is the same for both vertices:
an edge which comes from~$1$ is greater than an edge which come
form~$0$. In order to obtain the Morse transformation we
\emph{change the order of the symbols~$0,1$ depending on the
next symbol},
\[
0 \prec_0 1,\ 1 \prec_1 0
\]
as follows:
\[
\{x_i\} \prec \{y_i\} \ \ \Longleftrightarrow\ \ \exists\,j:\
x_i=y_i \ \mbox{for} \ i>j
\]
and
\[
\{x_j\} \prec_z y_j, \ \mbox{where} \ z=x_{j+1}=y_{j+1}.
\]
This allows us to define a new partial order as the
lexicographic order on the cofinal  set of paths, or points
of~$\textbf{Z}_2$ or set of all sequences of~$0,1$.

We now show what the next point is in the sense of the new
order. In order to define the next sequence to
the sequence~$x=(x_1,x_2,\dots)\in \textbf{Z}_2, \quad x_i=0
, 1,$ we use the following procedure.
\begin{enumerate}
\item Move through the sequence until the first repetition
    of symbols~$x_i=x_{i+1}$, this will be either~$\dots00$
    or~$\dots 11$;
\item Change~$00\rightarrow10$, or~$11 \rightarrow 01$;
\item Substitute all the previous digits to ones,~$11111
    \cdots1$ in the first case, and correspondingly to
zeros~$00000\cdots0$ in the second case:
\begin{eqnarray*}
M((01)^n00*)&=&(1^{2n+1} 0*),\\
M((01)^n1*)&=&(0^{2n}1*),\\
M((10)^n0*)&=&(1^{2n}0*),\\
M((10)^n11*)&=&(0^{2n+1} 1*).
\end{eqnarray*}
\end{enumerate}

For example,
\begin{eqnarray*}
M(00*)&=&10*,\\
M(0100*)&=&1110*,\\
M(011*)&=&001,\\
M(100*)&=&110,\\
M(1011)=0001
\end{eqnarray*}
The rule expressing how to go from~$x$ to~$M(x)$ may also be
expressed as follows. Suppose that in a sequence
\[
\textbf{Z}_2 \ni x =(x_1,x_2,\dots x_r,x_{r+1},\dots), x_i= 0,
 1,i=1 \dots
\]
the first two adjacent coordinates which are equal
are~$x_r=x_{r+1}$.

This means that~$x_1=x_3=x_5=\cdots =x_t \ne x_2=x_4=\dots
x_s$, where either~$t=r,s=r+1$ or~$t=r-1,s=r$ depending on the
parity of~$r$.

If~$x_r=0$, then the sequence~$M(x)$ has the first~$r-1$
coordinates equal to~$1$, and all subsequent coordinates are
not changed:~$(1,1 \dots 1,0,*,* )$. In this case, in the
formula~$M(x)=x+t(x)$, we have~$t(x)>0$.

If~$x_r=1$ then~$M(x)$ has the first~$r-1$ coordinates equal
to~$0$, and all subsequent coordinates are not
changed:~$(0,0,\dots 0,1,*,*\dots)$. In this case, in the
formula~$M(x)=x+t(x)$, we have~$t(x)<0$.

Our algorithm of definition of~$M(x)$ does not work if~$x$ has
finitely many~$0$s or $1$s or equal adjacent coordinates
with~$0$ and~$1$. Define as an \emph{exceptional set
of}~$\textbf{Z}_2$ the countable set of points
of~$\textbf{Z}_2$ whose dyadic decomposition has finitely many
adjacent pairs~$x_r,x_{r+1}$ of the type:~$x_r=x_{r+1}=0$
and~$x_r=x_{r+1}=1$, or (in a more direct description) the set
of sequences which are cofinal
to~$(0)^{\infty},(1)^{\infty},(01)^{\infty},(10)^{\infty}$.

\begin{theorem}\label{theoremA}
The orbit partition of the Morse transformation~$M$ in its adic
realization~$\pmod{0}$ coincides with the orbit partition of
the odometer. More exactly, consider the~$M$-orbit of the
point~$x\in \textbf{Z}_2$ which has infinitely many coordinates
with~$x_r=x_{r+1}$, and infinitely many~$0$s and~$1$s (this
occurs on a set of full measure). Then the orbit of~$x$
coincides with the set of all points which are cofinal
with~$x$. On the set of exceptional points (which has Haar
measure~$0$) the orbit equivalence relation of~$M$ is different
from the relation of cofinality.
\end{theorem}

\begin{proof}
We start the proof with a simple lemma.

\begin{lemma}\label{lemmaB}
Let~$x\in \textbf{Z}_2$ and assume that ~$x=\{x_1 \dots \}$ is
the dyadic expansion of~$x$. If for some~$r$ we
have~$x_r=x_{r+1}$, then each element~$y$ of~$\textbf{Z}_2$ for
which~$y_i=x_i, 1<i<r$ belongs to the~$M$-orbit of~$x$.
\end{lemma}

\begin{proof}
Suppose that~$x_r=0$. Then if we apply the definition of the
Morse transformation to the sequence with
coordinates
\[
x_1=x_2=\dots =x_n=0,
\]
we obtain all integers~$0,1 \dots 2^n-1$, so the fragment
of~$M$-orbit of the point~$x_k=0, k=1 \dots n$ coincides with
the set~$0,1\dots 2^r-1$, consequently all~$y$ with the
condition~$y_i=x_i, 1<i<r$ belongs to the~$M$-orbit of the
point~$\textbf{0}$. If~$x_r=1$ then the same is true if we
start with a sequence with coordinates
\[
x_1=x_2=\dots =x_n=1.
\]
\end{proof}

Returning to the proof of Theorem~\ref{theoremA}, we use the
condition that there are infinitely many~$k$ with~$x_k=1$ in
order to define correctly the full (two-sided) orbit of~$x$
(see the paragraph above about exceptional orbits). Now suppose
that two points~$x$ and $y$ are cofinal, and
\[
x_k=y_k, k>N.
\]
Then there exists~$r>N$ such that~$x_r=x_{r+1}=0$ so, by
Lemma~\ref{lemmaB},~$y$ belongs to the orbit of~$x$, completing
the proof.
\end{proof}

It is also possible to define the Morse transformation as adic transformation
using
\emph{differentiation} of binary sequences. Define the
differentiation operation
\[
D:{\bf Z}_2\to {\bf Z}_2
\]
by
\[
D(\{x_n\}_{n=0}^{\infty}) = \{(x_{n+1}-x_n)\ \mbox{\rm mod}\ 2,\ n=0,1, \dots\}.
\]
We remark that there are no good, simple ``arithmetic'' or
``analytic'' expressions for the behavior of~$D$. The next
result (see~\cite{VS}) relates the Morse--Thue sequence to the
operator~$D$.

\begin{lemma}
$T\circ D = D\circ M.$
\end{lemma}

This is an immediate corollary of the definition of~$M$. The
new definition was made by the author as an example of the adic
realization of the transformation (see~\cite{VL,V1}). In the
adic realization, the Morse transformation~$M$ is
a~$2$-covering of the odometer in its algebraic form.

\subsection{Jump function of the Morse automorphism, and Morse arithmetic}

Now we give a precise expression for the
``jump-function''~$t(\cdot)$ in the formula~$M(x)=x+t(x)$.
Notice that the value of the function~$t(x)$ depends on the
finite fragment of~$x$; more exactly, on the
fragment~$(x_1,\dots x_r), r=r(x)$ where~$r=\min\{k\mid
x_k=x_{k+1}\}$ (see above).

Define the sequence:
\[
a_r = \left\{ \begin{array}{cc} \frac{2^{r+1}-1}{3}\, & \mbox{if}\ r\equiv 1\ (\mbox{mod}\ 2), \\[1.4ex]
                                \frac{2^{r+1}-2}{3}\, & \mbox{if}\ r\equiv 0\ (\mbox{mod}\ 2)\end{array} \right.
\]
for~$r\geq 0$.

The first few values of~$a_r$ as a function of~$r=0,1 \dots$
are shown below:
\begin{center} \begin{tabular}{c|ccccccccccc}
$r$  & 0 & 1 & 2 & 3 &  4 &  5 &  6 &  7 &   8 &   9 & $\cdots$ \\
\hline $a_r$ & 0 & 1 & 2 & 5 & 10 & 21 & 42 & 85 & 170 & 341 &
$\cdots$ \end{tabular}
\end{center}

This sequence satisfies the recurrence relation
\[
a_{r+1}=2^r+a_{r-1}, \quad r=0,1,2 \dots ,
\]
with the initial conditions~$a_0=0,a_1=1$. Another evident
relation is
\[
a_{r-1}+a_r+1=2^r.
\]
We conclude that
\[
2^{r-1}\leq a_r<2^r,
\]
so in particular there can only be one number~$a_r$ in the
interval between two adjacent powers of~$2$.

The recurrence relation
\[
a_{2n}=2a_{2n-1}+1; \quad a_{2n+1}=2a_{2n},\quad n=1,2\dots;
\quad  a_0=0
\]
is a corollary of the definition. Notice that the dyadic
expansion ~$a_r$ corresponds to the sequence
\[
(01)^{\frac{r+1}{2}}\quad \mbox{for odd}\quad r>1
\]
and
\[
(10)^{\frac{r}{2}} \quad \mbox{for even}\quad r>0,
\]
or equivalently
\[
(01)^n = a_{2n+1}, \quad (10)^n =1(01)^{n-1}= a_{2n}.
\]

For each~$x \in \textbf{Z}_2$, we define~$r=r(x)$ to be the
minimal index of the coordinate for which the equality~$x_r=x_{r+1}$
occurs for the first time (see above). Consequently, for
all~$x$ which has a repetition~$00$ or~$11$ we have the
following result.

\begin{theorem}
\begin{enumerate}
\item The transformation~$M$ is defined on the
    group~$\textbf{Z}_2$ of~$2$-adic integers. It is
    continuous at all but two points:~$-\frac{1}{3}$
    and~$-\frac{2}{3}$. The transformation~$M$ preserves
    the Haar measure on~$\textbf{Z}_2$ and is metrically
    isomorphic to the Morse automorphism.
\item The explicit formula for the Morse transformation
    on~$x \in \textbf{Z}_2$ is
\[
M(x)=\left\{\begin{array}{ll}x+a_r, &\mbox{if}\quad x_r=0, \\
x-a_r, &\mbox{if}\quad x_r=1. \end{array}\right.
\]
where~$r=r(x)$.
\end{enumerate}
Both cases can be expressed in the formula
\[
M(x)=x+(-1)^{x_r}a_r,\quad \mbox{where} \quad r=r(x).
\]
If we compare this with the initial formula, we
obtain
\[
M(x)=T^{t(x)}x=x+t(x),
\]
so
\[
t(x)=M(x)-x= (-1)^{x_{r(x)}} a_{r(x)}.
\]
\end{theorem}

Our formula can be applied to the integers as elements of the
group of dyadic integers~${\mathbb Z} \subset\textbf{Z}_2$ as
follow. For~$n\in\mathbb Z$, denote by~$r(n)$ the minimal
number of the digit in the dyadic decomposition of
\[
n=\sum x_k 2^k, x_k=0, 1
\]
for which the value~$x_{r(n)}=x_{r(n)+1}= \epsilon(n)$. For
example,
\[
r(n)=1, \quad \epsilon(n)=0,\quad \mbox{if}\quad n\equiv 0
\pmod{4},
\]
\[
r(n)=1, \quad \epsilon(n)=1;\quad \mbox{if}\quad n\equiv 3
\pmod{4},
\]
\[
r(n)=2, \quad \epsilon(n)=0,\quad \mbox{if}\quad n\equiv 1
\pmod{8},
\]
\[
r(n)=2, \quad \epsilon(n)=1;\quad \mbox{if}\quad n\equiv 6
\pmod{8}.
\]
The general formula for~$n,k = 1,2 \dots$ is the following:
\[
r(n)=k, \quad \epsilon(n)=0, \quad\mbox{if}\quad n \equiv
a_{k-1} (\mbox{mod}\quad 2^{k+1})
\]
and
\[
r(n)=k, \quad \epsilon(n)=1, \quad \mbox{if} \quad n \equiv
-a_{k-1}-1 (\mbox{mod}\quad 2^{k+1}).
\]
The general formula for~$M(x)$, which generalizes the formula
above, is
\[
M(n)= n+(-1)^{\epsilon(n)}a_{r(n)}.
\]
Thus we obtain the \emph{Morse order on the integers}: the
following table illustrates this order on the
semigroup~$\mathbb N$ which is half the orbit of \textbf{0}:
\begin{center} \begin{tabular}{c|ccccccccccccccccc}
$n$ & 0 & 1 & 2 & 3 & 4 & 5  & 6 & 7 & 8 & 9 & 10 & 11 & 12 &
13 & 14 & 15  \\ \hline $ M(n) $ & 1 & 3 & 7 &  2 & 5 & 15 & 4
& 6 & 9 & 11 & 31 & 10 & 13 & 8 & 12 & 14
\end{tabular}
\end{center}
\medskip
The extension to negative integers is given by the relation
\[
M(-n)=-M(n-1)-1.
\]
This give the order on the other half orbit corresponding to
\textbf{1}:
\begin{center} \begin{tabular}{c|ccccccc}
-1 & -2 & -3 & -4 & -5 & -6 & -7  \\ \hline -2 & -4 & -8 & -3 &
-6 &-16 & -5  \end{tabular}
\end{center}
These tables give the time substitution at the
point~\textbf{0}. We call this the \emph{Morse order}.

The following sequences show us the Morse dynamics on the
integers, namely what is the Morse re-ordering of the integers:
$$
0\rightarrow 1\rightarrow 3 \rightarrow 2\rightarrow 7 \rightarrow 6 \rightarrow 4 \rightarrow 5\rightarrow 15 \cdots,$$
and for negative integers:
$$-1 \rightarrow -2 \rightarrow -4 \rightarrow -3 \rightarrow -8 \rightarrow -7 \rightarrow -5 \rightarrow-6 \cdots.$$

As we saw, the integers generate the exceptional
(semi)-orbits~$M$ (see above and Section~\ref{subsectA}); our
goal is to describe this re-ordering for generic orbits and to
present the time substitutions in a more explicit form.

\section{The structure of time substitutions; main construction}

Now we are ready to give the answer to the question about
general time substitutions for the Morse automorphism. The
first step is the definition of very interesting finite
permutations (elements of the groups~$S_{2^n}$) which will be
used to describe the re-ordering of the group~$\mathbb Z$, and
the time substitutions.

\subsection{Definition of Morse permutations and Morse order}

Our first construction concerns some special elements of the
finite symmetric groups~$S_{2^n}$ which we call \emph{Morse
permutations}, and defines a linear order on the set~$1,2,\dots
2^n$.

For all~$n>1$ we will define by induction a permutation~$g_n
\in S_{2^n}$. It is convenient for this definition to consider
the ordered set of integers:~$\{2^n,2^n+1, \dots ,2^{n+1}-1\}$,
instead of~$\{1,\dots 2^n\}$; later we will use it as
permutations of an arbitrary linear ordered set on
integers~$b,b+1,b+2, \dots, b+2^{n}-1, b\in \mathbb Z$ (using
the shift~$i \to i+b, i=2^n,\dots 2^{n+1}-1$).

Here are some of the first permutations (n=1,2,3) presented in
cycle form:
\begin{eqnarray*}
g_1&:&(2,3),\\
g_2&:&(4, 5, 7, 6),\\
g_3&:&(8, 9, 11, 10, 15, 14, 12, 13).
\end{eqnarray*}
We define the Morse permutations for arbitrary~$n$ by
induction. Suppose we have already defined~$g_{n-1}$, as a
permutation of the set
\[
\{2^{n-1}, \dots, 2^n-1\}.
\]
Then~$g_n$ will be defined as a permutation of the set
\[
\{2^n, \dots , 2^{n+1}-1\}.
\]
The action of~$g_n$ on the first half of the set~$\{2^n, \dots
, 2^n+2^{n-1}-1\}$ is defined as an ``almost'' shift of
permutation~$g_{n-1}$ onto~$2^n$. More precisely,
\[
g_n(i)=g_{n-1}(i-2^{n-1})+2^{n-1},\quad 2^n\leq i \leq 2^n+
2^{n-1}-1
\]
with one important exception:
\[
g_n(a_{n+1})=2^{n+1}-1
\]
(remember that~$2^n<a_{n+1}<2^{n+1}$). On the second
half~$\{2^n+2^{n-1}, \dots, 2^{n+1}-1\}$ the action of~$g_n$ is
also made of copies of the action of~$g_{n-1}$, but slightly
different:
\[
g_n(i)=g_{n-1}(i-2^n)+2^n,
\]
again with one exception:
\[
g_n(2^n+a_n)=2^n.
\]
All these permutations are well-defined, and are cyclic
permutations, nevertheless the first and second examined halves
of the set~$2^n\dots 2^{n+1}-1$ are ``almost invariant'' in the
sense that there is only one element of each half that has the
image in the opposite half.

We will use the Morse permutations in order to define the
linear order on the set. For this we need to \emph{break the
cycle at one point}. This point we will call the \emph{first}
(or the \emph{minimal}) point, and its pre-image (in the cycle)
as the \emph{last} (or the \emph{maximal}) point. We will write
the cycle from the first point when it is already selected.

We now define two opposite linear orders on the set
\[
\{2^n,2^n+1, \dots ,2^{n+1}-1\},
\]
breaking the cycle as follows. The order~$\tau_n$ has minimal
element~$2^{n+1}-1$ and the second order~${\bar \tau}_n$ has minimal
element~$2^n$. It is clear from the definition that the maximal
(or the last) element in the order~$\tau_n$ is~$a_{n+1}$, and
in the order~${\bar \tau}_n$ the maximal element
is~$2^{n+1}+2^n-a_{n+1}-1$. The order~${\bar \tau}_n$ is simply
the image of~$\tau_n$ under reflection~$i\leftrightarrow
2^{n+1}+2^n-i-1$. Recall that the symbol~$a\lessdot b$ means
that~$b$ is next to~$a$ in the sense of the order.

The structure of the Morse permutation and order will be more
transparent if we divide the set~$\{2^n, \dots , 2^{n+1}-1\}$
into groups of four elements. We will see that there are two
types~($\tau$ and~$\bar \tau$) of such groups which are
alternate.

\begin{example}
$$\tau_3:   \quad  15\lessdot 14\lessdot 12\lessdot 13\lessdot 8\lessdot 9\lessdot 11\lessdot 10,$$
$${\bar \tau}_3: \quad  8\lessdot 9\lessdot 11\lessdot 10\lessdot 15\lessdot 14\lessdot 12\lessdot 13.$$
\end{example}

\subsection{Random linear order on the group $\mathbb Z$, and time substitution for the Morse transformation}

We want to define an explicit linear order on the group~$\mathbb Z$
depending on~$x$ which corresponds to the time substitution
from the odometer to the Morse transformation:
\[
\cdots\negthinspace\rightarrow\negthinspace t(-2,x)\negthinspace \rightarrow
\negthinspace t(-1,x)\negthinspace\rightarrow\negthinspace
t(0,x)\negthinspace=\negthinspace0\negthinspace \rightarrow\negthinspace t(1,x)\negthinspace=
\negthinspace t(x)\negthinspace\rightarrow\negthinspace t(2,x)
\negthinspace\rightarrow
\negthinspace\cdots,
\]
where we recall that~$M^kx=T^{t(k,x)}$. The values of~$t(k,x)$
for a fixed generic~$x$ run over all of the group~$\mathbb Z$.
Thus we want to reorder the orbit of~$T$ to the orbit of~$M$.

\begin{definition}
The Morse random order~$\tau(x)$ (corresponding to the
non-exceptional point~$x$) on the group~$\mathbb Z$ is the
linear order defined using the map~$k \mapsto t(k,x)$
where~$M^kx=T^{t(k,x)}x$. In other words, it is the re-ordering
of the~$T$-order to the~$M$-order on the orbit of the
point~$x$.
\end{definition}

We will give an implicit description of the Morse order
(depending on~$x$). Firstly, we describe the structure of our
answer, and explain what the Morse linear order~$\tau(x)$ looks
like.

%%\textbf{General description of the Morse linear order on the
%%the group $\mathbb Z$}

\begin{definition}\label{definitionofmorselinearorder}
A \emph{Morse linear order~$\tau$} on the group~$\mathbb Z$
comprises the linear orders on the systems of countably many
finite intervals in~$\mathbb Z$, each of length a power of two,
the union of which is the whole group. On each of the finite
intervals, the linear order follows the Morse order defined
above, and we glue the boundary elements (the maximal and
minimal points of the adjacent intervals).
\end{definition}

Such an ordering, and such a corresponding infinite
permutation, belongs to the class of \emph{locally finite
linear orderings of~$\mathbb Z$} defined in
Section~\ref{sectiononlocallyfiniteorderings}. The order
depends on the point~$x$ and is therefore called a \emph{random
order}.

The locally finite ordering has a system of increasing
intervals on~$\mathbb Z$ of length~$2^k$ for various~$k$, and
we equip each of these with its Morse order. At each stage it appears
that the old interval is included in the new interval. The final
points of each linear order~$\tau_n$ are glued to one of the
boundary points of the next interval. The length and order of
gluing hardly depends on the non-exceptional point~$x$, indeed
the structure of the construction is universal.

We next describe in more detail such a construction.

\subsection{A parametrization of the points in $\textbf{Z}_2$}

It is convenient to parameterize the points~$x\in \textbf{Z}_2$
as follow. Let~$x=(x_1,x_2 \dots)$ be the dyadic decomposition
of~$x$. The sequence of coordinates is a sequence of
independent~$0,1$ variables with
probability~$(\frac{1}{2},\frac{1}{2})$. Instead of the
coordinates of~$x$, we consider all the numbers~$r_1(x),r_2(x)
\cdots$ of coordinates for which~$x_{r_i}=x_{r_i+1}$, and fix
also the value~$0$ or~$1$ of~$x_{r_1}=\epsilon(x)$. For almost
all elements~$x$ the sequence~$\{r_n\}$ is an infinite
increasing sequence, which together with~$\epsilon(x)=0, 1$
defines~$x$ uniquely. Indeed, it is easy to see that if we
know~$\{r_n\}$ and~$\epsilon(x)$ then we can restore all the
subsequent coordinates\footnote{The sequence~$\{r_n(x)\}$ can
be interpreted easily in terms of the differentiation
operation~$D$ from
Section~\ref{sectiontraditionaldefinitionofmorse},
since~$r_n(x)$ is the place on which the sequence~$Dx$ has
a~$1$.}. The probabilistic properties of the parameters can be
obtained from the fact that the Haar measure is the Bernoulli
measure.

\begin{lemma}\label{geometricaldistributionlemma}
\begin{enumerate}
\item The differences~$r_{n+1}-r_n$ are mutually
    independent and have the same geometrical
    distribution
\[
    Prob\{r_n=k\}=2^{-k}, k=1,\dots.
    \]
\item Consider the function~$t(x)$ from
    Section~\ref{sectionongeneraltimesubstitutions}
    with~$Mx=T^{t(x)}x$. Then
\[
Prob\{t(x)=\pm a_r\}\equiv Prob \{\Theta(x)(\textbf{1})
=
\pm a_r\}=Prob\{r_1(x)=r\}=\textstyle\frac{1}{2^r}.
\]
\end{enumerate}
\end{lemma}

This means that the values of~$t(x)$ are not arbitrary, and
have exponentially decreasing probability.

Fix~$x$; for each such choice, corresponding
data~$\{r_n(x)=r_n\}_{n \in {\mathbb N}}$,
and~$\epsilon(x)=\epsilon$, we make a corresponding re-ordering
of the group~$\mathbb Z$. The point~$x$ will correspond to~$0$
in the group~$\mathbb Z$, and the ordinary order on the
group~$\mathbb Z$ corresponds to the dynamics of the
odometer~$k \backsim T^k x$.

\subsection{Construction of the Morse linear order, and its time-substitution.}

Now we describe the algorithm which sequentially constructs,
for each fragment of dyadic numbers, the final interval
on~$\mathbb Z$ with the needed linear order.

\begin{enumerate}
\item The initial interval is constructed as follows.
    Consider~$r_1(x)$, which is greater than or equal
    to~$1$ by definition. If~$r_1=1$, then~$x=(0,0
    *,*\dots)$ or~$x=(1,1,*,* \dots )$. In the first
    case,~$Mx=Tx=(1,0,*,*\dots$), so the~$M$-image of~$x$
    (of~$0$) is the same as the~$T$-image and is~$1$,
    thus~$1$ is~$M$-next to~$0$. In the second
    case~$Mx=T^{-1}x$, so that~$-1$ is~$M$-next to~$0$.
    Thus the initial interval is either~$\{0,1\}$ with the
    usual linear order, or~$\{-1,0\}$ with that linear
    order. Suppose now that~$r_1=r>1$. In this case the
    initial~$(r-1)$-fragment of~$x$ has
    coordinates~$(01)^s, (10)^s,1(01)^s$ or~$0(10)^s$,
    where~$s=\frac{r-1}{2}$ or~$s=\frac{r-2}{2}$ depending
    on the parity of~$r$. We will consider~$2^r$ points
    from~$\textbf{Z}_2$ whose coordinates with
    indices~$m>r$ are the same as the coordinates of~$x$.
    The set of these points in~$\mathbb Z$ is the
    interval~$0,1,\dots 2^r$, and we will shift this set in
    order that the fragment~$x_1,\dots x_r$ of the
    point~$x$ starts on the place of~$0$ in~$\mathbb Z$.
    We
    translate the interval of integers~$\{0,1 \dots
    2^r-1\}\subset\mathbb Z$ to the interval of integers
  \[
  \bar I_1=\{-a_{r-1}, -a_{r-1}+1, \dots 0,\dots,
  2^r-a_{r-1}-1=a_r\}
  \]
   if~$x_r=0$, or
\[
   {\bar I}_1=\{-a_r, -a_r+1,\dots, 0,\dots,
   2^r-a_r-1=a_{r-1}\}
\]
if~$x_r=1$. These are the initial intervals of our
     construction, namely
\[
     I_1=(-a_{r-1},a_r)
\]
and
\[
{\bar I}_1=(-a_r,a_{r-1}).
\]
In the construction we have only used the
fragment
\[
(x_1,x_2,\dots x_{r+1})
\]
of the point~$x$. We must define a new linear order on this
set according to the action of the Morse transformation.
\item Now we apply the Morse order~$\tau_r$ on the
    interval~$I_1$ and~${\bar\tau}_r$ on the
    interval~${\bar I}_1$ defined above, using the shift
    of~$2^r,2^{r+1}-1$ to those two intervals, as mentioned
    in the previous item. We obtain a linear order on the
    intervals depending on whether~$x_r$ is~$0$ or~$1$.
    Carrying the boundary points from the interval
\[
    2^n,\dots, 2^{n+1}-1
\]
    to the interval which we obtained gives the following
    boundary points of our linear order. Initial (minimal)
    points are
\[
    -a_{r-1} \in(I_1,\tau(r)),
\]
    correspondingly
\[
    a_{r-1}\in({\bar I}_1),{\bar\tau}(r)),
\]
    and the maximal (or last) points are
\[
    a_{r-1}+1 \in(I_1,\tau(r)),
\]
    correspondingly
\[
    -a_{r-1}-1 \in({\bar I}_1,{\bar\tau}(r)).
\]
    We have obtained the initial step of the construction
of the linear order~$\tau(x)$.
\end{enumerate}

\begin{example}
Let~$r_1=3$. Then~$x=(0,1,0,0,*,\dots)$
or~$x=(1,0,1,1,*\dots)$, and one of the ends of the interval
will be~$a_r=5$, and we obtain the interval
\[
(-2,-1,0,1,2,3,4,5)
\]
for the first case and
\[
(-5,-4,-3,-2,-1,0,1,2)
\]
for the second case. The initial points are~$-2$ and~$+2$ and
the last points are~$3$ and~$-3$ correspondingly. The linear
orders are~$\tau$ and~$\bar\tau$, namely
\[
-2\rightarrow -1\rightarrow 1\rightarrow  0\rightarrow
5\rightarrow 4\rightarrow 2\rightarrow 3,
\]
or
\[
2\rightarrow 1 \rightarrow -1\rightarrow 0\rightarrow -5\rightarrow -4\rightarrow -2\rightarrow -3.
\]
\end{example}

We have given an explicit form of the first part of the time
substitution for each point, so we already have an algorithm
for calculating the function $t(\cdot)$. Because this function
uniquely defined the functions~$t(k,\cdot)$ by the formula
\[
t(k,x)=\sum_{i=1}^{k-1}t(M^i x),
\]
it is possible to stop the algorithm for constructing the time
substitution here. However we will give a continuation in order
to describe it as a random re-ordering of the whole
group~$\mathbb Z$.

\begin{enumerate}
\item[(3)] The general inductive step may be described as
    follows. Suppose that we have already considered the
    first~$n-1$ members of the sequence~$r_{n-1}(x)\equiv
    r_{n-1}$, and obtained the linear order of one of the
    intervals~$I_{n-1}(x)=\{b_{n-1},c_{n-1}\}$ of
    length~$2^{r_{n-1}}$ and including~$0\in\mathbb Z$
    which is equipped with the Morse linear
    order~$\tau_{r_{n-1}}$ or~$\bar\tau_{r_{n-1}}$ with the
    minimal point of the order coincided with one of the
    endpoints of~$I_n$, either~$b_{n-1}$ or~$c_{n-1}$
    correspondingly. We will choose the next
    interval~$I_n(x)=\{b_n,c_n\}$ and define a linear order
    on it so as to include~$I_{n-1}(x)$, and such that the
    restriction of the linear order on~$I_{n-1}(x)$
    coincides with the initial linear order. Consider the
    number~$r_n$, being the next after~$r_{n-1}$ with equal
    coordinates~$x_{r_n}=x_{r_n +1}$. Denote the maximal
    point of the order on~$I_{n-1}$ by~$l_{n-1}$. There are
    two cases:
    \begin{enumerate}
\item If~$r_n=r_{n-1}+1$, then~$x_{r_n}=x_{r_{n+1}}$
    and in this case~$I_n=I_{n-1}\cup J$, the minimal
    element of~$I_n$ is the same as in~$I_{n-1}$,
    where~$J$ is an interval adjoining~$I_{n-1}$ from
    the side which is opposite to the minimal element.
    The linear order on~$I_n$ has the same type~$\tau$
    or~$\bar\tau$ as on~$I_{n-1}$. The next element to
    the maximal element of~$I_{n-1}$ in~$I_n$ will be
    the second (non-minimal) endpoint of~$I_n$.
\item If~$r_n>r_{n-1}+1$ then the construction depends
    on the parity of the difference~$r_n-r_{n-1}>1$. If
    this difference is odd, then~$x_{r_{n-1}}\ne
    x_{r_n}$ and the Morse order on~$I_n$ will change
    its type to be opposite to the type of
    order~$I_{n-1}$; in particular the minimal element
    of~$I_n$, say~$b_n$, will be on the opposite side
    to the minimal element~$c_n$ of~$I_{n-1}$. If the
    difference is even, then minimal elements are
    both~$b$ or both~$c$. Now the interval~$I_n=J\cup
    I_{n-1}\cup J'$ where~$J,J'$ are adjacent to
    the~$I_{n-1}$ intervals in~$\mathbb Z$. The lengths
    of~$J$ and of~$J'$ are equal to
\[
    |J|=\sum_{i:r_{n-1}+1<i<r_n-1;i\equiv r_{n-1}(2)}
    2^i;
\]
and
\[
|J'|=\sum_{i:r_{n-1}+1<i<r_n-1;i\equiv
   1+r_{n-1}(2)} 2^i.
\]
It is clear that the length of~$I_n$ is~$2^{r_n}$. By
definition, the restriction of the Morse order on~$I_n$
onto the interval~$I_{n-1}$ coincides with the initial
order on~$I_{n-1}$. The next element to the maximal
element of~$I_{n-1}$ in~$I_n$ will be the second
(nonminimal) endpoint of the interval~$I_n$.
\end{enumerate}
\end{enumerate}

The randomness of the construction and of the re-ordering
consists in the various possibilities for the
sequence~$r_n(x)$ and the values~$x_{r_n}$. Thus the
re-ordering can be different for various values of~$x$.
Nevertheless the structure of the new orders are similar
for all points. Now it is evident that the probabilistic
behavior of the length of the intervals~$I_n$ depends
precisely on the sequence~$r_n$ and the size of jumps has
geometrical distribution (see
Lemma~\ref{geometricaldistributionlemma}). For example, the
long jumps have exponentially small probability.

\subsubsection{An exercise and an informal explanation}

A good concrete example of the ordering of~$\mathbb Z$ is given
by rational (periodic) elements of~$\textbf{Z}_2$. We will give
the first fragment of the linear order for\footnote{This
equality is true
because
\[
\underbrace{(100)^{\infty}+\cdots+(100)^\infty}_{7{\rm {\
}terms}}=(1)^{\infty}=-1.
\]
Similar arguments give the
equalities~$(01)^{\infty}=-\frac{2}{3},(10)^{\infty}=-\frac{1}{3},(1100)^{\infty}=-\frac{1}{5}$
and so on.}
\[
x=(100)^{\infty}=-\textstyle\frac{1}{7}.
\]
 and leave for the reader the case of
\[
 x=(1100)^{\infty}=-\textstyle\frac{1}{5}.
\]
The orbit of the point~$-\frac{1}{7}=(100)^{\infty}$ is
interesting because it is a generic point (in the sense that
the orbit is of type~$\mathbb Z$), and at the same time the
values of the coordinates which are repeated are the same,
so~$x_{r_k}=0$. In the case of~$(1100)^{\infty}$ the situation is
more complicated: the values of~$x_{r_k}$ change as~$x_{r_{2k+1}}=1,
x_{r_{2k}}=0$. The periodicity does not give any simplifications,
so together both cases give the full picture.

This is the beginning: the invariant~$32$ digits of the linear
order for~$-\frac{1}{7}=(100)^{\infty}$ are
\[
\{\dots |\lessdot-9 \lessdot-8 \lessdot -6 \lessdot -7,| \lessdot -2 \lessdot -3 \lessdot -5 \lessdot -4|,
\lessdot +6 \lessdot +5 \lessdot +3 \lessdot +4|
$$
$$\lessdot -1,\lessdot 0 \lessdot 2,\lessdot 1|, \lessdot22 \lessdot 21\lessdot 19\lessdot 20|,\lessdot 15\lessdot 16\lessdot 18 \lessdot 17|,\lessdot$$
$$\lessdot 7\lessdot 8\lessdot 10 \lessdot 9|\lessdot 14 \lessdot 13\lessdot 11
\lessdot 12|\dots\}.
\]
Here~$a\lessdot b$ means that~$Mx=T^a x, M^2x=T^bx$, for example $MT^4=T^6$ and~$MT^6=T^5$.

We can see that in all the examples the order subdivided (as
marked by~$|$) into the blocks with~$4$ points with
order~$(1,2,4,3)$ or~$(4,3,1,2)$, then the~$4$ blocks generate
the block of the next level, and so on. But the distance
between the quadruples (or jumps) depends on~$x$, more exactly
on the number~$k$ of adjacent coordinates which have the same
values~$r_k=r_{k+1}$.

The cases~$x= (0)^\infty, (1)^\infty, (01)^\infty, (10)^\infty$
remain. As we will see in the next section, these are
one-sided: the first two left-sided, and the second two
right-sided.

\subsection{Addendum: Exceptional orbits}
The points
\[
\textbf{0}=(0)^{\infty},\textbf{-1}=(1)^{\infty}
\]
are exceptional: they have no full orbits because they have no
pre-images. So the semi-orbit of~$\textbf{0}$ defines a linear
order (and a permutation) on the semigroup~${\mathbb Z}_+$,
and the semi-orbit of~$\textbf{1}$ similarly defines a linear
order on~${\mathbb Z}_-$ (see the formulas at the end of
Section~\ref{subsectA}). The points with denominator~$3$ are
also exceptional and not generic, as we saw. The orbit of the
point~$(10)^{\infty}$ does not coincide with an orbit of the
odometer. Now we can obtain the complete comparison of the
orbit partition on~$\textbf{Z}_2$ for the odometer~$T$ and the
Morse transformation~$M$. We remark that for the odometer the
points~$\textbf{0}=(0)^{\infty}$ and~$\textbf{1}=(1)^{\infty}$
belong to one orbit, but for the Morse transformation this is
not true.

Consider the following orbits of the odometer~$T$:
\begin{eqnarray*}
\tau_0,&&\mbox{ the orbit of }\textbf{0};\\
\tau_1,&&\mbox{ the orbit of }\textstyle\frac{3k+1}{3}, k \in \mathbb Z;\\
\tau_2,&&\mbox{ the orbit of }\textstyle\frac{3k+2}{3}, k\in \mathbb Z.
\end{eqnarray*}
On the other side, the Morse transformation~$M$ has two
positive semi-orbits (which have no past), namely the orbits of
the points~$0^{\infty}$ and~$1^{\infty}$, and two negative
semi-orbits (which have no future), namely the orbits of the
points~$(01)^{\infty}$ and~$(10)^{\infty}$. We call them
\emph{maximal points}, and denote the set of these two points
as~$MAX=\{ (10)^{\infty},(01)^{\infty}\}$. We will join those
three orbits of the odometer~$T$ with four semi-orbits of~$M$
and obtain two new orbits of the Morse automorphism~$M$ by
definition. By the initial definition, the orbit~$\tau_0$
divides into two positive semi-orbits of~$M$ -- one starts
with~$0$ and the second with~$-1$: each of these glue
correspondingly with orbits~$\tau_1$ and~$\tau_2$ of~$T$, and
recall that for the initial definition of~$M$ they are only
negative semi-orbits. So we glue each two positive semi-orbits
to two negative semi-orbits, and obtain two new full orbits
of~$M$.

Because the transformation~$M$ is not defined on these two
sequences, we can by definition choose the values among another
two sequences which conversely have no \emph{pre-images}, or
for which the inverse map~$M^{-1}$ is not defined. There are
only two such points~$z$ for which there is no~$x$ with the
property~$M(x)=z$, namely~$(0)^\infty=0\in\mathbb Z$
and~$(1)^\infty=-1\in \mathbb Z$.

We may assume\footnote{We glued the semi-orbit of~$\textbf{0}$
to the semi-orbit of~$-\frac{1}{3}$, and the semi-orbit
of~$\textbf{1}$ to~$-\frac{2}{3}$, but we can change this
gluing to the opposite one.} that
\[
M((10)^\infty)\equiv M(-\textstyle\frac{1}{3})=(0)^{\infty} =\textbf{0},\
\ \ M((01)^{\infty}\equiv M(-\textstyle\frac{2}{3})=
(1)^\infty=-\textbf{1}.
\]
Here the boldface numbers denote rational
integers:~$\textbf{0},\textbf{1}\in \mathbb Z\subset
\textbf{Z}_2$.

This gives us the following picture:
\[
\cdots
+\textstyle\frac{7}{3},+\frac{10}{3},+\frac{16}{3},+\frac{13}{3},-\frac{17}{3},-\frac{14}{3},
\]
\[
-\textstyle\frac{8}{3},-\frac{11}{3},+\frac{4}{3},
+\frac{1}{3},-\frac{5}{3},-\frac{2}{3}, \quad (!) \quad
-1,-2,-4.-3, -8\dots;
\]
the second is finished with~$-\frac{1}{3}$; and we prolonged it
with~$0$:
\[
\cdots-\textstyle\frac{10}{3},+\frac{16}{3},+\frac{14}{3},+\frac{11}{3},+\frac{5}{3},+\frac{8}{3},
\]
\[
-\textstyle\frac{7}{3},-\frac{4}{3},+\frac{2}{3}, -\frac{1}{3} \quad (!)
\quad  0, 1, 3, 2, 7\dots.
\]
As predecessors to the symbol~$!$ in both pictures we have the
former orbits of~$T$, which now become the negative semi-orbits
of~$M$, which we glued to integers. Thus the transformation~$M$
is now defined on the whole group~$\textbf{Z}_2$. In
particular, we have defined~$M$ for all integers from~$\mathbb
Z$.

It is clear from the definition that
\[
M(-n)=-M(n-1)-1, \quad n=1,\dots.
\]
This formula is valid for all~$x \in \textbf{Z}_2$
since~$M(-x)=-M(x-1)-1$, and is true even for the exceptional
points~$x=-\frac{1}{3}$ and $x=-\frac{2}{3}$ since
\[
M(-\textstyle\frac{1}{3})=-M(-\frac{2}{3})-1.
\]
It is easy to deduce from the definition of~$M$ that~$M$ is
continuous on~$\textbf{Z}_2 \setminus MAX$, and that it is not
possible to extend~$M$ by continuity to those two points,
because, for example, the limit of each of the two
sequences~$(10)^n(0)^{\infty}$ and~$(10)^n(1)^{\infty}$ as~$n$
tends to infinity, is the same,
namely~$(10)^\infty(=-\frac{1}{3})$, but the values of~$M$ on
the sequences tends in the first case
to~$(1)^{\infty}(=-\textbf{1})$ and in the second case
to~$(0)^{\infty}(=\textbf{0})$.

Thus, except for three orbits of~$T$ (or four semi-orbits
of~$M$), all the other orbits are simultaneously orbits of both
transformation. This completely defines the adic realization of
the Morse transformation and its orbit partition, as well as
the time substitution of the odometer.

\section{Conclusion}
We gave an explicit form of the time change on the orbits of
the odometer in order to obtain the Morse transformation. The
answer shows us that the structure of the new ordering of the
orbit is \emph{locally finite} in the sense we have defined.
We proved that Morse transformation and odometer are allied in the
sense of our definition (section 2). This structure
of the time change indicates that the two automorphisms are not
very different in terms of their orbits (but nevertheless can
have different spectrum). It is possible that they have the
same \emph{entropy scale} in the sense of~\cite{V2}. What can
we say more generally about the properties of automorphisms
which are related by such a time change? For example, if one is
Bernoulli will the second also be Bernoulli?

The measure on the space of linear orders of~$\mathbb Z$ (or on
the space of locally finite permutations of~$\mathbb Z$) which
we have defined with our algorithm is of great interest itself.
It is the image of the invariant measure on~$\textbf{Z}_2$ with
respect to the function~$S\{t(k,\cdot)\}_{k\in \mathbb
Z}:\textbf{Z}_2 \rightarrow {\frak S}_{\mathbb Z}$.

It is interesting also to study the time change for other
examples of measure-preserving automorphisms, for example the
other substitutions like the Morse transformation, and
automorphism with positive entropy. The classes of random
infinite permutations which appear in these cases give new
examples of nontrivial measures on the infinite symmetric
group.

\end{document}